\documentclass[11pt, usletter, reqno, twoside, makeidx]{amsart}
\usepackage[margin=2.8cm, marginpar=2.3cm]{geometry}

\usepackage[color=blue!20!white,textsize=tiny]{todonotes}
\usepackage{bm}
\usepackage{multicol}
\usepackage{bm,amsmath,amsthm,amssymb,mathtools,verbatim,amsfonts,tikz-cd,mathrsfs}
\usepackage{mathrsfs}
\usepackage{amsmath} 
\usepackage{amsthm}
\usepackage{amssymb} 
\usepackage{amsfonts}
\usepackage{tikz-cd}
\usepackage{graphicx}
\usepackage{xcolor}
\definecolor{darkgreen}{rgb}{0,0.5,0}
\usepackage[normalem]{ulem}
\usepackage{comment}
\usepackage{caption,subcaption,pinlabel}
\usepackage{graphics}
\usepackage[pagebackref=true]{hyperref}
\usepackage{scalerel}
\usepackage{comment}
\usepackage{enumitem}
\usepackage{mathtools}
\usepackage{tikz}
\usetikzlibrary{matrix,arrows,decorations}
\usepackage{mathabx}
\usetikzlibrary{matrix}
\usepackage{scrextend}

\newtheorem{theorem}{Theorem}[section]
\newtheorem{lemma}[theorem]{Lemma}

\newtheorem{question}[theorem]{Question}

\newtheorem{corollary}[theorem]{Corollary}

\theoremstyle{definition}
\newtheorem{definition}[theorem]{Definition}

\newtheorem{remark}[theorem]{Remark}

\newcommand{\im}{\textrm{im}\hspace{0.05cm}}

\def\spinc {{\operatorname{spin^c}}}
\def\Spinc {{\operatorname{Spin^c}}}

\def\s{\mathfrak s}

\def\CF {\mathit{CF}}
\def\HF {\mathit{HF}}

\newcommand \CFm {\CF^-}
\newcommand \HFm {\HF^-}

\newcommand{\Z}{\mathbb{Z}}
\newcommand{\Q}{\mathbb{Q}}

\newcommand{\F}{\mathbb{F}}

\newcommand{\Arf}{\mathrm{Arf}}

\let\int\relax
\newcommand{\int}{\mathring}

\usepackage{accents}

\DeclareMathSymbol{\wtilde}{\mathord}{largesymbols}{"65}

\mathchardef\mhyphen="2D

\newcommand{\CFK}{\mathcal{CFK}}

\newcommand{\gr}{\text{gr}}

\newcommand{\id}{\mathrm{id}}

\newcommand{\Cconn}{C_{\mathrm{conn}}}

\newcommand{\scU}{\mathscr{U}}
\newcommand{\scV}{\mathscr{V}}

\newcommand{\cC}{\mathcal{C}}

\newcommand{\cF}{\mathcal{F}}

\newcommand{\cU}{\mathscr{U}}
\newcommand{\cV}{\mathscr{V}}

\newcommand{\Cyl}{\mathrm{Cyl}_{\phi, \iota_K}}

\renewcommand{\d}{\partial}

\newcommand{\bF}{\mathbb{F}}

\newcommand{\tl}{t_{\lambda}}
\newcommand{\tm}{t_{\mu}}
\newcommand{\Ds}{\mathcal{S}}

\title{Gompf's cork and Heegaard Floer homology}

\author{Irving Dai}
\address{Department of Mathematics\\The University of Texas at Austin\\ Austin, TX, USA}
\email{irving.dai@math.utexas.edu}

\author{Abhishek Mallick}
\address{Department of Mathematics\\Rutgers University - New Brunswick\\ NJ, USA}
\email{abhishek.mallick@rutgers.edu}

\author{Ian Zemke}
\address{Department of Mathematics\\Princeton University\\  Princeton, NJ, USA}
\email{izemke@math.princeton.edu}

\begin{document}
\maketitle
\begin{abstract}
Gompf showed that for $K$ in a certain family of double-twist knots, the swallow-follow operation makes $1/n$-surgery on $K \# -K$ into a cork boundary. We derive a general Floer-theoretic condition on $K$ under which this is the case. Our formalism allows us to produce many further examples of corks, partially answering a question of Gompf. Unlike Gompf's method, our proof does not rely on any closed 4-manifold invariants or effective embeddings, and also generalizes to other diffeomorphisms.
\end{abstract}

\maketitle

\section{Introduction}\label{sec:1}
The study of exotic phenomena has traditionally occupied a central role in the development of low-dimensional topology. Following the work of Akbulut \cite{Akbulut}, it has emerged that this has a close connection to the theory of corks. Recall that a \textit{cork} is a compact, contractible $4$-manifold $C$ equipped with a boundary diffeomorphism $f \colon \partial C \rightarrow \partial C$ which does not extend over $C$ as a diffeomorphism. In contrast, such an $f$ always extends over $C$ as a homeomorphism by work of Freedman \cite{Freedman}. It is now known that any two smooth structures on the same simply-connected closed 4-manifold are related by cutting out some $C$ and re-gluing via $f$, an operation called a \textit{cork twist} \cite{Matveyev, CFHS}.

In \cite{gompf2017infinite}, Gompf gave a simple new construction leading to the first instance of an \textit{infinite-order cork}, or \textit{$\Z$-cork}. This is a compact, contractible $4$-manifold $C$ equipped with a boundary diffeomorphism $f \colon \partial C \rightarrow \partial C$ such that no power of $f$ extends over $C$ as a diffeomorphism. Gompf's cork is constructed by considering the $3$-manifold
\[
Y_{K, m} = S^3_{1/m}(K \# -K)
\]
for $K$ a knot in $S^3$ and $m \in \Z^{\neq 0}$. It is not hard to see that $Y_{K, m}$ bounds the contractible manifold $C_{K, m}$ obtained by extending $1/m$-surgery on $K \# -K$ over its standard ribbon disk in $B^4$. Note that $K \# - K$ admits a self-isotopy defined by pushing the summand $K$ along $K \# -K$ once around in a full loop. This is referred to as the \textit{swallow-follow operation} $\tl$; we denote the induced self-diffeomorphism on the surgered manifold $Y_{K, m}$ also by $\tl$. See Section~\ref{sec:2.1} for further discussion.

Gompf showed that for a specific family of double-twist knots $K$ beginning with $K = 4_1$, each $C_{K, m}$ may be embedded in a blown-up elliptic surface $X$ such that twists by powers of $\tl$ correspond to different Fintushel-Stern knot surgeries, and in fact give an infinite collection of pairwise distinct smooth structures on $X$. This proves that these $(C_{K, m}, \tl)$ are $\Z$-corks. It is natural to ask which other choices of $K$ make $(C_{K, m}, \tl)$ into a $\Z$-cork, or even just a cork. This question was posed in \cite{gompf2017infinite}:

\begin{question}[{\cite[Question 1.6]{gompf2017infinite}}]\label{q:1.1}
Let $m \in \Z^{\neq 0}$ and $\tl$ be induced from the swallow-follow operation on $K \# -K$. For which knots $K$ is $(C_{K, m}, \tl)$ is a $\Z$-cork?
\end{question}

In this paper, we investigate the question of when $(C_{K, m}, \tl)$ is a cork, although Gompf originally posed Question~\ref{q:1.1} in the setting of $\Z$-corks. (We expect that the methods of this paper can be generalized to establish infinite-order corks; see Remark~\ref{rem:1.4} below.) As far as the authors are aware, the knots considered in \cite{gompf2017infinite} are the only affirmative examples of such $(C_{K, m}, \tl)$ appearing in the literature, even in the weaker setting where $(C_{K, m}, \tl)$ is only required to be a (regular) cork. It was shown by Ray-Ruberman that if $K$ is a torus knot, then $(C_{K, m}, \tl)$ is \textit{not} a cork for any $m \in \Z^{\neq 0}$ \cite{RayRuberman}. The question of which $K$ satisfy Question~\ref{q:1.1} thus certainly appears to be subtle. 

Gompf's original proof relies on finding an embedding of $C_{K, m}$ in a closed $4$-manifold $X$ and identifying powers of the cork twist with different Fintushel-Stern knot surgeries on $X$. The authors are not aware of any systematic method for establishing a similar construction for other families of knots. In this paper, we instead provide a flexible criterion on the knot Floer homology of $K$ which guarantees that $(C_{K, m}, \tl)$ is a cork. The perspective we take is also slightly different than the one in Question~\ref{q:1.1}: instead of corks, we focus on the notion of a \emph{strong cork}, due to Lin-Ruberman-Saveliev \cite{LRS}. Recall that a strong cork is a pair $(Y, f)$ where $Y$ is a 3-manifold and $f$ is a diffeomorphism of $Y$ which does not extend over \emph{any} homology ball that $Y$ bounds.\footnote{We generally require $Y$ to bound at least one contractible manifold, so that a strong cork is a cork.} 

In the present work, we show that $(Y_{K, m}, \tl)$ constitutes a strong cork for a large family of knots $K$, including many of the double-twist knots from \cite{gompf2017infinite}. In the context of Question~\ref{q:1.1}, this means that the role of the specific manifold $C_{K, m}$ is de-emphasized: we may replace $C_{K, m}$ by \textit{any} contractible manifold (or homology ball) that $Y_{K, m}$ bounds. For instance, in the construction of $C_{K, m}$, we may use any slice disk for $K \# -K$ in place of the standard one. Note that $Y_{K, m}$ may also bound a contractible manifold (or homology ball) which is not constructed from a slice disk in such a manner.

As we discuss in Section~\ref{sec:2.1}, the swallow-follow operation is a longitudinal Dehn twist along an incompressible torus $T$ in $Y_{K, m}$, and it is natural to ask whether any other Dehn twists along $T$ make $C_{K, m}$ into a cork. In \cite{Gompfhandle}, it was shown that the meridional twist extends over $C_{K, m}$ for any $K$ and $m = \pm 1$. While we are not able to establish any examples of corks formed from meridional twists, the approach of this paper is suited to studying general Dehn twists along $T$. Our formalism can likewise be used to produce examples of corks constructed using a more flexible class of boundary diffeomorphisms than Dehn twists along $T$; see Section~\ref{sec:2.2}.

\subsection{Statement of results}\label{sec:1.1}

We now state our results. In Section~\ref{sec:2.5}, we define a Floer-theoretic condition on $K$ which we call \textit{$\Ds$-nontriviality}. For experts, this means that the Sarkar map $s$ is homotopically nontrivial on the $(\iota_K$-$)$connected complex of $K$. Denote the longitudinal twist by $\tl$ and the meridional twist by $\tm$. We prove:

\begin{theorem}\label{thm:1.2}
If $K$ is $\Ds$-nontrivial, then $(Y_{K, m}, \tl^i \tm^j)$ is a strong cork for all $(m, i, j) \in \Z^3$ with $m$ and $i$ both odd.
\end{theorem}

The condition of being $\Ds$-nontrivial is quite mild, as the following computation indicates:

\begin{corollary}\label{cor:1.3}
Let $K$ be a Floer-thin knot satisfying
\begin{equation}\label{eq:1.1}
2 \Arf(K) + |\tau(K)| \equiv 1 \text{ or } 2 \bmod 4.\footnote{For experts, this simply means that the number of ``box subcomplexes" in the local equivalence class of $K$ is odd. See \cite[Section 8]{HM}.}
\end{equation}
Then $(Y_{K, m}, \tl^i \tm^j)$ is a strong cork for all $(m, i, j) \in \Z^3$ with $m$ and $i$ both odd.
\end{corollary}

See for example \cite{petkovacables} for a discussion of Floer-thin knots. Note that since all alternating or quasi-alternating knots are Floer-thin, Corollary~\ref{cor:1.3} greatly expands the set of corks arising from Gompf's construction. All double-twist knots $\kappa(r, -s)$ (with $r$ and $s$ positive) considered in \cite{gompf2017infinite} are alternating. It is not hard to check that $\kappa(r, -s)$ satisfies \eqref{eq:1.1} precisely when $r$ and $s$ are both odd; this includes the simplest example of $K = 4_1$. In order to emphasize the flexibility of our approach, we list the set of knots with eight or fewer crossings to which Theorem~\ref{thm:1.2} (via Corollary~\ref{cor:1.3}) applies. These are:
\[
4_1, 5_2, 6_3, 7_4, 7_5, 7_7, 8_1, 8_2, 8_6, 8_7, 8_{12}, 8_{13}, 8_{14}, 8_{15}, 8_{17}, 8_{18}, \text{and } 8_{21}.
\]
Of these, only $4_1$ and $8_1$ appear in \cite{gompf2017infinite}, illustrating the wide applicability of Theorem~\ref{thm:1.2}. (On the other hand, the knots $6_1$ and $8_3$ are covered by \cite{gompf2017infinite} but are not included in Theorem~\ref{thm:1.2}.)

We are also able to use Theorem~\ref{thm:1.2} to produce examples of corks arising from connected sums of torus knots. This is particularly interesting in light of the results of \cite{Gompfhandle} and \cite{RayRuberman}, which show that if $K$ is a torus knot, then $C_{K, \pm 1}$ does not constitute a cork for any twist in $H_1(T, \Z)$. In contrast, we prove that applying Gompf's construction to the connected sum of torus knots often produces a cork. Indeed, it is straightforward to check that Corollary~\ref{cor:1.3} even applies to the simplest connected sum of torus knots $K = T_{2, 3} \# T_{2, 3}$. More generally, in Corollaries~\ref{cor:trefoil_connected} and \ref{cor:nonthin} we show that Theorem~\ref{thm:1.2} applies to the families
\[
K = T_{2, 2n + 1} \# T_{2, 2n +1} \quad \text{and} \quad K = -2T_{2n, 2n+1} \# T_{2n, 4n+1}
\]
for $n$ odd. Note the latter class of knots is \textit{not} Floer-thin. In general, the condition of $\Ds$-nontriviality is fairly mild and can be verified for many non-thin knots; see Section~\ref{sec:2.5} for further discussion.

We again emphasize that our approach to Question~\ref{q:1.1} is rather different than the one in \cite{gompf2017infinite} and does not consist of finding embeddings of corks into specific closed $4$-manifolds. Instead, we proceed by analyzing the induced action of $\tl$ (and $\tm$) on the Heegaard Floer homology of $Y_{K, m}$. This action is defined due to the work of Juh\'asz-Thurston-Zemke \cite{JTZ} regarding the action of the mapping class group on Heegaard Floer homology. Such ideas were first used to study branched double covers of knots by Alfieri-Kang-Stipsicz \cite{AKS}. A systematic application to corks was carried out by Dai-Hedden-Mallick \cite{DHM}; see also the work of Lin-Ruberman-Saveliev using monopole Floer homology \cite{LRS}.

\begin{remark}\label{rem:1.4}
The authors expect that the ideas of the present paper can likely be strengthened to show that the corks in Theorem~\ref{thm:1.2} are infinite-order. At the moment, however, there are certain technical obstructions to doing this. As a first step, we would need to obtain an appropriate set of naturality results for Heegaard Floer theory with $\Z$-coefficients, together with a definition of the Floer cobordism maps over $\Z$; see \cite{Gartner} for progress in this direction.
\end{remark}

Our methods can also be extended to analyze a more general class of self-diffeomorphisms defined on knot surgeries, which we describe in Section~\ref{sec:2.2}. Indeed, fix any knot $K$ in $S^3$ and let $\phi$ be a relative self-diffeomorphism of $(S^3, K)$. This induces a self-diffeomorphism of any surgered manifold $\smash{S^3_{1/n}(K)}$, which by abuse of notation we also denote by $\phi$. In Section~\ref{sec:3}, we describe a sufficient condition for the pair $\smash{(S^3_{1/m}(K), \phi)}$ to be a strong cork in terms of the local equivalence class of the triple $(\CFK(K), \phi, \iota_K)$. In Section~\ref{sec:2.6}, we define an integer-valued Fr\o yshov-type invariant
\[
\delta(K, \phi) \ge 0
\]
which may be computed from $\CFK(K)$ (with the actions of $\phi$ and $\iota_K$) and completely characterizes the existence of a local map from the trivial complex into $(\CFK(K), \phi, \iota_K)$. We prove:
\begin{theorem}
\label{thm:numerical-intro}
 If $\delta(K, \phi)>0$, then $(S^3_{1/m}(K), \phi)$ is a strong cork for any $m$ positive and odd.
\end{theorem}

Gompf's construction is obtained by taking $\phi$ to be the swallow-follow operation on the connected sum $K \# -K$. In fact, the swallow-follow operation fits into a larger family of relative self-diffeomorphisms on composite knots of the form $K_1 \# K_2$; we call these split diffeomorphisms. Split diffeomorphisms are especially convenient from the point of view of Floer theory; further examples are given in Section~\ref{sec:5}.

The results of this paper may also be viewed as a generalization of the program of \cite{DHM} in the following sense: in \cite{DHM}, Floer-theoretic techniques were used to produce many novel families of (strong) corks via $1/n$-surgeries on classes of symmetric slice knots. This is in contrast to previous constructions of corks in the literature, which have generally focused on explicit handle decompositions of candidate contractible manifolds. Corollary~\ref{cor:1.3} vastly enlarges the set of (strong) corks arising as surgery on slice knots, with the swallow-follow operation playing the role of the knot symmetry in \cite{DHM}. Compare \cite[Theorem 1.11]{DHM}. We note that $\delta$ differs from the invariants defined in \cite{DHM} and that the examples presented here cannot be recovered from the formalism of \cite{DHM}; see Remark~\ref{rem:2.6}.

\noindent
\subsection*{Organization} In Section~\ref{sec:2}, we review the algebraic setup of Heegaard Floer homology and define the notion of $\Ds$-nontriviality. In Section~\ref{sec:3}, we prove a general cork-theoretic detection result for certain knot surgeries and use this to establish Theorem~\ref{thm:numerical-intro}. In Section~\ref{sec:4}, we apply this to prove Theorem~\ref{thm:1.2}. In Section~\ref{sec:5}, we prove Corollary~\ref{cor:1.3} and give further examples of strong corks detected using our obstructions. 


\subsection*{Acknowledgments} The authors would like to thank Kristen Hendricks, Jen Hom, Tye Lidman, and Maggie Miller for helpful conversations. ID was partially supported by NSF grant DMS-2303823. AM was partially supported by NSF grant DMS-2019396. IZ was partially supported by NSF grant DMS-2204375.

\section{Background}\label{sec:2}

In this section, we give a more precise definition of Gompf's cork and review some essential features of Heegaard Floer and knot Floer homology.

\subsection{Gompf's construction}\label{sec:2.1}
Let $K_1$ and $K_2$ be any pair of knots in $S^3$. Define a self-diffeomorphism $\tl$ of $(S^3, K_1)$ as follows: denote the boundary of a tubular neighborhood of $K_1$ by $T$ and let $T \times [-1, 1]$ be a neighborhood of $T$ which does not intersect $K_1$. On $T \times [-1, 1]$, define $\tl$ to be the trace  of an isotopy which rotates $T$ once around so that a point on $T$ sweeps out an oriented longitude of $K_1$. On the complement of $T \times [-1, 1]$, define $\tl$ to be the identity. We refer to $\tl$ as the \textit{longitudinal twist}; the \textit{meridional twist} $\tm$ is defined similarly. Note that $\tl$ and $\tm$ fix $K_1$ pointwise.

Now form the connected sum $K_1 \# K_2$ by placing $K_2$ in a small ball which is disjoint from $[-1, 1] \times T$ and on the same side of $T$ as $K_1$. Then $\tl$ and $\tm$ also define self-diffeomorphisms of the pair $(S^3, K_1 \# K_2)$, which we likewise denote by $\tl$ and $\tm$. Since these similarly fix $K_1 \# K_2$ pointwise, they induce self-diffeomorphisms of any surgered manifold $S^3_r(K_1 \# K_2)$. Abusing notation, we again denote these surgered diffeomorphisms by $\tl$ and $\tm$. See Figure~\ref{fig:2.1}. These diffeomorphisms were referred to as the \textit{torus-twists} by Gompf \cite{gompf2017infinite}.


\begin{figure}[h!]
\center
\includegraphics[scale=0.6]{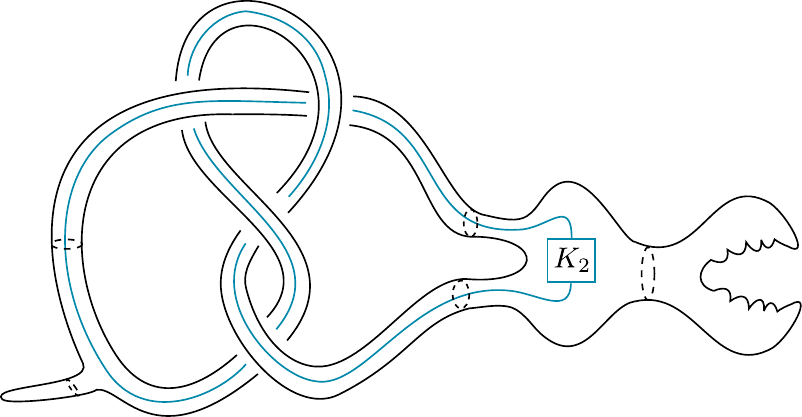}
\caption{The \textit{swallow-follow torus} $T$ in the case where $K_1 = 4_1$. The torus is following $K_1$ while swallowing $K_2$.}
\label{fig:2.1}
\end{figure}
Now suppose $K_2$ is inverse to $K_1$ in the concordance group, so that $K_1 \# K_2$ bounds a slice disk $D$. For any $m \in \Z^{\neq 0}$, define $C_{D, m}$ by cutting out $D$ from $B^4$ and attaching a $2$-handle along a meridian of $K_1 \# K_2$ with framing $-m$. It can be checked that $C_{D, m}$ is a contractible manifold and that
\[
S^3_{1/m}(K_1 \# K_2) = \partial C_{D, m}.
\]
Note that different choices of $D$ will in general give different contractible manifolds. Importantly, it is \textit{not} clear whether or not the self-diffeomorphisms $\tl$ and $\tm$ extend over $C_{D, m}$, or whether or not this fact is independent of $D$. Moreover, $\smash{S^3_{1/m}(K_1 \# K_2)}$ may bound contractible manifolds or homology balls which are not constructed from a slice disk in such a fashion. 

Gompf's construction is obtained by specializing to the case where $K_1 = K$, $K_2 = -K$, and $D$ is the standard ribbon disk for $K \# - K$, in which case we denote $C_{D, m}$ by $C_{K, m}$.\footnote{Our conventions differ slightly from \cite{gompf2017infinite}: Gompf's notation $C(\kappa, m)$ has boundary given by $-1/m$-surgery on $K \# -K$, and thus corresponds to our $C_{K, -m}$.} In \cite[Theorem 1.2]{gompf2017infinite}, it is shown that if $K$ is taken from a certain family of double-twist knots $\kappa(r, -s)$, then $C_{K, m}$ can be embedded in a blown-up elliptic surface such that cork twists by powers of $\tl$ give pairwise nondiffeomorphic $4$-manifolds (distinguished by the Seiberg-Witten invariants). In particular, no power of $\tl$ extends as a diffeomorphism over $C_{K, m}$. 

\subsection{Generalizing Gompf's construction}\label{sec:2.2}
Although we will primarily be interested in the swallow-follow diffeomorphism, Gompf's construction can be placed in a more general context as follows:

\begin{definition}\label{def:2.1}
A \textit{relative (self-)diffeomorphism of $(S^3, K)$} is an orientation-preserving self-diffeomorphism of $(S^3, K)$ which fixes a neighborhood $N(K)$ of $K$ pointwise. If $\phi$ is a relative diffeomorphism of $(S^3, K)$, then $\phi$ induces a self-diffeomorphism of any surgery along $K$ by choosing the surgery solid torus to lie in $N(K)$. By abuse of notation, we denote the resulting diffeomorphism again by $\phi$ and refer to it as the corresponding \textit{surgered diffeomorphism}.
\end{definition}

In this paper, we will be interested in a particular class of relative diffeomorphisms:

\begin{definition}\label{def:2.2}
Let $\phi_1$ and $\phi_2$ be relative diffeomorphisms of $(S^3, K_1)$ and $(S^3, K_2)$, respectively. Let $B_1$ be a small ball intersecting $K_1$ which is fixed by $\phi_1$, and similarly for $B_2$. We obtain a relative diffeomorphism $\phi_1 \# \phi_2$ of $(S^3, K_1 \# K_2)$ by forming the connected sum $(S^3, K_1) \# (S^3, K_2)$ along these balls. We refer to a self-diffeomorphism constructed in this manner as a \textit{split diffeomorphism}.
\end{definition}

The swallow-follow diffeomorphism is the split diffeomorphism $\tl \# \id$ on $K \# - K$ obtained by putting the longitudinal twist $\tl$ on the first factor and the identity on the second. (In the context of Gompf's construction, we often write $\tl$ in place of $\tl \# \id$ when our meaning is clear.) We give further examples of split diffeomorphisms in Section~\ref{sec:5}. In general, if $K$ is slice, then any relative diffeomorphism $\phi$ of $(S^3, K)$ gives rise to a candidate family of strong corks by considering the surgeries $\smash{(S^3_{1/m}(K), \phi)}$. If $K = K_1 \# K_2$, the sliceness condition can of course be tautologically manufactured by choosing $K_2$ to be a concordance inverse of $K_1$.

As we will see, another reason for considering the class of split diffeomorphisms is that the action of $\phi_1 \# \phi_2$ on $\CFK(K_1 \# K_2)$ is straightforward to understand. Indeed, as the name suggests, the action of $\phi_1 \# \phi_2$ on $\CFK(K_1 \# K_2)$ may be identified with the action of $\phi_1 \otimes \phi_2$ on $\CFK(K_1) \otimes \CFK(K_2)$. This will allow us to formulate a more concise Floer-theoretic condition for detecting corks.

\subsection{Heegaard Floer homology}\label{sec:2.3}
We now give a brief overview of the background in Heegaard Floer homology necessary for our proof. We assume that the reader has a broad familiarity with the Heegaard Floer package \cite{OS3manifolds1, OS3manifolds2} as well as a general understanding of the involutive Heegaard Floer formalism of Hendricks-Manolescu \cite{HM} and Hendricks-Manolescu-Zemke \cite{HMZ}.

Let $Y$ be a rational homology sphere and $\s$ be a self-conjugate $\spinc$-structure on $Y$. There are two automorphisms of $\CFm(Y, \s)$ that we consider in this paper. Firstly, in \cite{HM} Hendricks and Manolescu defined the Heegaard Floer involution $\iota$:
\[
\iota: \CFm(Y, \s) \rightarrow \CFm(Y, \s).
\]
This is a grading-preserving, $\F[U]$-equivariant homotopy involution on $\CFm(Y, \s)$. Secondly, suppose that $Y$ is equipped with a self-diffeomorphism $\phi$. For the sake of brevity, we will often refer to $(Y, \phi)$ as an \textit{equivariant (rational) homology sphere}. By work of Juh\'asz-Thurston-Zemke \cite{JTZ}, for each $\spinc$-structure $\s$ on $Y$ such that $\phi_*(\s) = \s$, we obtain an induced action
\[
\phi: \CFm(Y, \s) \rightarrow \CFm(Y, \s),
\] 
which by abuse of notation we also denote by $\phi$. This is a grading-preserving, $\F[U]$-equivariant chain map from $\CFm(Y, \s)$ to itself.\footnote{Strictly speaking, $\phi$ should be an element of the based mapping class group. However, if $Y$ is a rational homology sphere, then it follows from \cite[Theorem D]{Zemkegraphcobord} that this condition can be relaxed; see for example \cite[Lemma 4.1]{DHM}.} Note that the action of $\phi$ has a homotopy inverse given by the action of $\phi^{-1}$. It is straightforward to show that $\iota$ and $\phi$ homotopy commute; see for example \cite[Lemma 4.4]{DHM}.

We formalize this information in the following abstract definition:

\begin{definition}\label{def:2.3}
A \textit{$(\phi, \iota)$-complex} consists of the following:
\begin{enumerate}
\item A free, finitely-generated, $\Q$-graded chain complex $C$ over $\F[U]$ such that 
\[
U^{-1}H_*(C) \cong \F[U, U^{-1}]. 
\]
We require $C$ to be graded by a coset of $\Z$ in $\Q$ with $\deg(\partial) = -1$ and $\deg(U) = -2$.
\item Grading-preserving, $\F[U]$-equivariant chain maps $\phi \colon C \rightarrow C$ and $\iota \colon C \rightarrow C$ such that $\phi$ admits a homotopy inverse, $\iota$ is a homotopy involution, and $\phi$ and $\iota$ commute up to homotopy.
\end{enumerate}
A \textit{morphism} (or \textit{map}) $f$ from $(C_1, \phi_1, \iota_1)$ to $(C_2, \phi_2, \iota_2)$ is a grading-preserving, $\F[U]$-equivariant chain map from $C_1$ to $C_2$ such that $f \phi_1 \simeq \phi_2 f$ and $f \iota_1 \simeq \iota_2 f$. A \textit{homotopy equivalence} of $(\phi, \iota)$-complexes consists of a pair of morphisms $f$ and $g$ between them that are homotopy inverses. We denote chain homotopy by $\simeq$. 
\end{definition}

If $\phi$ is a relative diffeomorphism of $(S^3, K)$, then $\phi$ acts trivially on the homology of the complement of $K$. It follows that the surgered diffeomorphism acts as the identity on the set of $\spinc$-structures on any surgery along $K$. Hence for any self-conjugate $\spinc$-structure $\s$, the triple
\[
(\CFm(S^3_r(K), \s), \phi, \iota)
\]
is a $(\phi, \iota)$-complex in the sense of Definition~\ref{def:2.3}.

\begin{definition}\label{def:2.4}
Let $f \colon C_1 \rightarrow C_2$ be a morphism of $(\phi, \iota)$-complexes. We say that $f$ is \textit{local} if the induced map 
\[
f_* \colon U^{-1} H_*(C_1) \cong \F[U, U^{-1}] \rightarrow U^{-1}H_*(C_2) \cong \F[U, U^{-1}]
\]
is an isomorphism. If there are local maps in both directions between $C_1$ and $C_2$, then we say that $C_1$ and $C_2$ are \textit{locally equivalent}. Occasionally, we will refer to $f$ as local even if it is only grading-homogeneous (rather than grading-preserving), so long as it homotopy commutes with $\phi$ and $\iota$ and satisfies the localization condition of Definition~\ref{def:2.4}.
\end{definition}

The importance of Definition~\ref{def:2.4} is given by the following simple lemma. We say that a cobordism $W$ from $(Y_1, \phi_1)$ to $(Y_2, \phi_2)$ is \textit{equivariant} if there exists a self-diffeomorphism $\phi$ of $W$ which restricts to $\phi_i$ on $Y_i$.

\begin{lemma}\label{lem:2.5}
Let $(W, \phi)$ be an equivariant negative-definite cobordism with $b_1(W) = 0$ between equivariant homology spheres $(Y_1, \phi_1)$ and $(Y_2, \phi_2)$. Suppose that there exists a self-conjugate $\spinc$-structure $\s$ on $W$ such that $\phi_*(\s) = \s$. Then
\[
F_{W, \s} \colon (\CFm(Y_1, \s|_{Y_1}), \phi_1, \iota_1) \rightarrow (\CFm(Y_2, \s|_{Y_2}), \phi_2, \iota_2)
\]
is a local map, up to grading shift.
\end{lemma}
\begin{proof}
It is a standard fact that $F_{W, \s}$ induces an isomorphism on $U^{-1} \HFm$ and that it satisfies the relation $F_{W, \s} \circ \iota_1 \simeq \iota_2 \circ F_{W,s}$. The fact that $F_{W, \s} \circ \phi_1 \simeq \phi_2 \circ F_{W,s}$ follows from \cite[Theorem A]{Zemkegraphcobord}; see for example \cite[Proposition 4.10]{DHM}.
\end{proof}

In order to establish that a given pair $(Y, \phi)$ is a strong cork, it thus suffices to prove that there is no local equivalence between the complex of $S^3$, which is given by $(\F[U], \id, \id)$, and the complex $(\CFm(Y), \phi, \iota)$. 


\begin{remark}\label{rem:2.6}
One can form a \textit{local equivalence group} by taking the set of all $(\phi, \iota)$-complexes and quotienting out by the notion of local equivalence. In the context of involutive Heegaard Floer homology, the notion of local equivalence first appeared in \cite{HMZ} and was subsequently utilized in \cite{DHM} to study corks and symmetries of manifolds. For experts, we note that Definition~\ref{def:2.3} qualitatively differs from previous such constructions by simultaneously including \textit{two} automorphisms of $C$. Indeed, one obtains coarser local equivalence groups by considering only $(C, \iota)$ or $(C, \phi)$, or even $(C, \iota \phi)$ as in \cite{DHM}. However, none of these suffice to capture the nontriviality of the examples in this paper. The first example of this nature was observed in \cite{dai20222}, where the above formalism is implicit.
\end{remark}

\subsection{Knot Floer homology}\label{sec:2.4}

We assume that the reader is familiar with the interpretation of knot Floer homology as a free, finitely generated chain complex $\CFK(K)$ over $\bF[\scU,\scV]$. See e.g. \cite{zemke2019link}. Given any such complex, there are maps
\[
\Phi = \dfrac{d}{d \cU}(\partial) \quad \text{and} \quad \Psi = \dfrac{d}{d \cV}(\partial);
\]
see \cite[Section 3]{Zemkequasistab}. We define the \textit{Sarkar map} to be
\[
s = \mathrm{id} + \Phi \Psi.
\]
This was studied in the context of the basepoint-moving action on $\CFK(K)$ \cite{sarkar2015moving, Zemkequasistab}. As in the case of $3$-manifolds, Hendricks and Manolescu \cite{HM} defined a knot Floer map:
\[
\iota_K \colon \CFK(K) \rightarrow \CFK(K).
\]
This is a skew-graded, skew-equivariant map. Suppose moreover that $\phi$ is a relative diffeomorphism of $(S^3, K)$. By work of Juh\'asz-Thurston-Zemke \cite{JTZ}, we again obtain an induced action
\[
\phi \colon \CFK(K) \rightarrow \CFK(K)
\]
which we also denote by $\phi$. This is a grading-preserving $\F[\cU, \cV]$-equivariant map. It is straightforward to check that $\phi$ and $\iota_K$ homotopy commute. We formalize the structure of $\CFK(K)$ in the following definition:

\begin{definition}\label{def:2.7}
A \textit{$(\phi, \iota_K)$-complex} consist of the following:
\begin{enumerate}
\item  A free, finitely-generated, bigraded chain complex $C$ over $\mathbb{F}[\cU, \cV]$ such that we have 
\[
(\cU,\cV)^{-1}H_{*}(C) \cong (\cU,\cV)^{-1} \mathbb{F}[\cU, \cV]. 
\]
We denote the bigrading $\gr = (\gr_{\cU}, \gr_{\cV})$ We require $\deg(\partial) = (-1, -1)$, $\deg(\cU)=(-2,0)$, and $\deg(\cV)=(0,-2)$.
\item A grading-preserving, $\F[\cU, \cV]$-equivariant chain map $\phi \colon C \rightarrow C$ and a skew-graded, skew $\mathbb{F}[\cU,\cV]$-equivariant chain map $\iota_K: C \rightarrow C$ such that $\iota^{2}_K \simeq s = \mathrm{id} + \Phi \Psi$. We require that $\phi$ have a homotopy inverse and that $\phi$ and $\iota_K$ homotopy commute.
\end{enumerate}
A \textit{morphism} (or \textit{map}) $f$ from $(C_1, \phi_1, \iota_{K_1})$ to $(C_2, \phi_2, \iota_{K_2})$ is a grading-preserving, $\F[\cU, \cV]$-equivariant chain map from $C_1$ to $C_2$ such that $f \phi_1 \simeq \phi_2 f$ and $f \iota_{K_1} \simeq \iota_{K_2} f$. A \textit{homotopy equivalence} of $\iota_K$-complexes consists of a pair of morphisms $f$ and $g$ between them that are homotopy inverses. We denote homotopy equivalence by $\simeq$.
\end{definition}

As before, we have:

\begin{definition}\label{def:2.8}
Let $f \colon C_1 \rightarrow C_2$ be a morphism of $(\phi, \iota_K)$-complexes. We say that $f$ is \textit{local} if the induced map 
\[
f_* \colon (\cU,\cV)^{-1}H_{*}(C_1) \cong (\cU,\cV)^{-1} \mathbb{F}[\cU, \cV] \rightarrow (\cU,\cV)^{-1}H_{*}(C_2) \cong (\cU,\cV)^{-1} \mathbb{F}[\cU, \cV]
\] 
is an isomorphism. If there are local maps in both directions between $C_1$ and $C_2$, then we say that $C_1$ and $C_2$ are \textit{locally equivalent}. 
\end{definition}

Definitions~\ref{def:2.7} and \ref{def:2.8} can of course be repeated in the absence of a self-diffeomorphism $\phi$. (This is equivalent to setting $\phi = \id$ throughout.) Doing so recovers the notion of an $\iota_K$-complex as defined in \cite{Zemkeconnected}.

\subsection{$\Ds$-nontriviality}\label{sec:2.5}

We now define the notion of $\Ds$-nontriviality. For this, we recall the work of Hendricks-Hom-Lidman \cite{HHL} regarding the \textit{connected complex}; see also \cite[Section 5]{HKPS}. Roughly speaking, this should be thought of as the simplest representative of the local equivalence class of an $\iota_K$-complex $(C, \iota_K)$.

Let $\cC = (C, \iota_K)$ be an $\iota_K$-complex. We call a local map $f$ from $\cC$ to itself a \textit{self-local map}. One can define a pre-order $\lesssim$ on the set of self-local maps by declaring $f \lesssim g$ if $\mathrm{ker}f \subseteq \mathrm{ker}g$. A self-local map $f$ is \textit{maximal} if for any other self-local map $g$ with $f \lesssim g$, we must have $g \lesssim f$. In \cite[Lemma 3.4]{HHL} it is shown that if $f$ is a maximal self-local map, then $f|_{\im f} \colon \im f \rightarrow C$ is injective. Hence we may define $(\iota_K)_f \colon \im f \rightarrow \im f$ by 
\[
(\iota_K)_f = f \circ \iota_K \circ (f|_{\im f})^{-1}.
\]
It is easily checked that the pair $(\im f, (\iota_K)_f)$ is an $\iota_K$-complex. The same proof as in \cite[Lemma 3.8]{HHL} shows that the chain isomorphism class of $(\im f, (\iota_K)_f)$ is independent of the choice of maximal self-local map $f$.

\begin{definition}\label{def:2.9}
Let $\cC = (C, \iota_K)$ be an $\iota_K$-complex. We define the $\iota_K$-\textit{connected complex} to be (the homotopy equivalence class of) 
\[
\cC_{\mathrm{conn}} = (\Cconn, \iota_{\mathrm{conn}}) = (\im f, (\iota_K)_f)
\]
for any maximal self-local map $f$. Note that the maps $(f|_{\im f})^{-1} \colon \Cconn  \rightarrow C$ and $f \colon C \rightarrow \Cconn$ are local equivalences of $\iota_K$-complexes.
\end{definition}

We note an important observation that will be helpful later on: 

\begin{lemma}\label{lem:2.10}
Let $\cC = (C, \iota_K)$ be an $\iota_K$-complex. Any self-local map $h \colon \Cconn \rightarrow \Cconn$ is a chain isomorphism.
\end{lemma}
\begin{proof}
Let $\Cconn = \im f$ for some maximal self-local map $f$ of $C$. Then $(f|_{\im f})^{-1} \circ h \circ f$ is a self-local map of $C$. If $\ker h \neq 0$, then this would have kernel a strict superset of $\ker f$, violating the maximality of $f$.
\end{proof}

\begin{definition}\label{def:2.11}
Let $\cC = (C, \iota_K)$ be a $\iota_K$-complex. We say $\cC$ is $\Ds$\textit{-nontrivial} if
\[
s \colon \Cconn \rightarrow \Cconn
\]
satisfies $s \not\simeq \mathrm{id}$. (Here, we view $\Cconn$ as an $\iota_K$-complex in its own right and define $s=\mathrm{id} + \Phi \Psi$ as in Section~\ref{sec:2.4}.)
\end{definition}

Although in general the $\iota_K$-connected complex of $K$ is difficult to compute, there are many classes of knots for which $\Cconn$ is understood. For instance, for Floer-thin knots, this computation is essentially contained in \cite[Propsition 8.1]{HM}. Once $\cC_\mathrm{conn}$ is determined, it is straightforward to calculate $s$ and decide whether $K$ is $\Ds$-nontrivial.

\subsection{Numerical invariants}\label{sec:2.6}

In this section, we define a numerical invariant which completely captures the existence of local maps from the trivial complex. Let $\cC = (C, \phi, \iota_K)$ be a $(\phi, \iota_K)$-complex. Recall that $A_0(\cC)$ is the subcomplex of $C$ spanned by all elements $x$ with $\gr_{\scU}(x) = \gr_{\scV}(x)$. This may be viewed as a singly-graded complex over the ring $\F[U]$, with the grading given by $\gr_{\scU} = \gr_{\scV}$ and $U = \scU \scV$. We denote the result by $A_0(\cC) = (A_0(C), \phi, \iota_K)$; it is immediate that $A_0(\cC)$ is a $(\phi, \iota)$-complex in the sense of Definition~\ref{def:2.3}. 

We may also define a chain complex $\Cyl(\cC)$, given by the total complex of the following diagram:

\begin{center}
\begin{picture}(200,50)
\put(0,20){$L_{\phi,\iota_K}(\mathcal{C}) = $}
\put(65,40){$A_0(C)$}
\put(130,40){$A_0(C)[-1]$}
\put(55,0){$A_0(C)[-1]$}
\put(100,40){$\xrightarrow{1+\iota_K}$}
\put(75,30){\rotatebox{270}{$\xrightarrow{\text{\rotatebox{90}{$1+\phi$}}}$}}
\end{picture}
\end{center}
For our purposes, it is helpful to have the natural map
\[
q\colon \Cyl(\cC) \to A_0(C)
\]
given by projecting onto the unshifted (i.e., top-left corner) copy of $A_0(C)$.

\begin{definition}
Let $\cC = (C, \phi, \iota_K)$ be a $(\phi, \iota_K)$-complex. We define $\delta(\cC)\in \Z$ to be 
\[
\delta(\cC)=-\frac{1}{2}\max\{ \gr(x): x\in H_*(\Cyl(\cC)) \text{ and }  q_*(x) \text{ is }  \bF[U]\text{-nontorsion}\}.
\]
Here, we will consider only complexes $(C, \phi, \iota_K)$ such that $C/\scU$ and $C/\scV$ are both homotopy equivalent to $\bF[\scV]$ and $\bF[\scU]$, where $1$ is given degree zero. We say that such complexes are of \emph{$S^3$-type}. It is easily checked that for complexes of $S^3$-type, $\delta(\cC) \geq 0$.
\end{definition}

The following lemma shows that $\delta$ completely characterizes the existence of local maps from the trivial complex into $A_0(\cC)$:

\begin{lemma}\label{lem:d=0->implies-local-map}
Let $\cC = (C, \phi, \iota_K)$ be a $(\phi, \iota_K)$-complex of $S^3$-type. Then $\delta(\cC)=0$ if and only if there is a local map from $(\bF[U], \id, \id)$ to $A_0(\cC)$.
\end{lemma}
\begin{proof} Note that a cycle in $\Cyl(\cC)$ consists of a triple $(x,y,z)$ such that
\[
\d x=0 \quad \d y=(1+\phi)(x)\quad \text{and} \quad \d z=(1+\iota_K)(x)
\]
where $x,y,z\in A_0(C)$. We observe that $q_*(x,y,z)=x$, so $\delta(\cC)=0$ if and only if there is a cycle $(x,y,z)$ in $\Cyl(\cC)$ such that $[x]$ is $U$-nontorsion in $H_*(A_0(C))$ and $\gr(x)= 0$. We may assume also that $\gr(y)=\gr(z)=1$ (here we think of $y$ and $z$ as elements of $A_0(\mathcal{C})$ as opposed to $\Cyl$). If $\delta(\cC) = 0$, we can thus define a local map $F$ from $(\F[U], \id, \id)$ to $(C, \phi, \iota_K)$ by setting $F(1)=x$. Note that 
\[
F\circ \id+\iota_K\circ F=[\d, h]\quad \text{and} \quad F\circ \id+\phi\circ F=[\d, j]
\]
where $h(1)=z$ and $j(1)=y$.  Similarly, given such a local map, we may construct such a cycle $(x,y,z)$, completing the proof. 
\end{proof}

In Lemma~\ref{lem:fromA_0toC} below, we show that there is a local map from $(\bF[U], \id, \id)$ to $A_0(\cC)$ (in the sense of Definition~\ref{def:2.4}) if and only if there is a local map from $(\F[\cU, \cV], \id, \iota_0)$ to $\cC$ (in the sense of Definition~\ref{def:2.8}). Here $\iota_0$ is the unique skew-graded, skew $\F{[\cU,\cV]}$-equivariant self-map of $\F[\cU, \cV]$. Hence $\delta$ characterizes local maps from the trivial complex in both the knot Floer and the large surgery settings.

We make the definition:

\begin{definition}
Let $K$ be a knot in $S^3$ and $\phi$ be a relative diffeomorphism of $(S^3, K)$. Define
\[
\delta(K,\phi)=\delta(\CFK(K), \phi, \iota_K).
\]
\end{definition}


\section{The General Obstruction}\label{sec:3}
We begin with a general cork detection result in the setting where $\phi$ is a relative diffeomorphism of $(S^3, K)$. The main claim of this section is the following:

\begin{theorem}\label{thm:3.1}
Let $\phi$ be a relative diffeomorphism of $(S^3, K)$. Suppose there is no local map
\[
(\F[U], \id, \id) \rightarrow (A_0(K), \phi, \iota_K).
\]
Then $(S^3_{1/m}(K), \phi)$ is a strong cork for any $m$ positive and odd.\footnote{Here, we implicitly suppose that $\smash{S^3_{1/m}(K)}$ bounds a contractible manifold, so that the non-extendability of $\phi$ is interesting.}
\end{theorem}

Note that this immediately gives the proof of Theorem~\ref{thm:numerical-intro}:

\begin{proof}[Proof of Theorem~\ref{thm:numerical-intro}]
Follows immediately from Theorem~\ref{thm:3.1} and Lemma~\ref{lem:d=0->implies-local-map}.
\end{proof}

We caution the reader that Theorems~\ref{thm:3.1} and \ref{thm:numerical-intro} have a restriction on the sign of $m$. To deal with negative $m$, note that $\smash{(S_{1/m}(K), \phi)}$ is a strong cork if and only if $\smash{(S_{-1/m}(-K), - \phi)}$ is a strong cork. The case of general $m$ may thus obtained by considering both $\delta(K, \phi)$ and $\delta(-K, -\phi)$. 

As discussed in Subsection~\ref{sec:2.3}, in order to show that $\smash{(S^3_{1/m}, \phi)}$ is a strong cork, it suffices to prove there is no local map
\[
(\F[U], \id, \id) \rightarrow (\CFm(S^3_{1/m}(K)), \phi, \iota).
\]
This is almost Theorem~\ref{thm:3.1}, but it is not quite the same. Indeed, Theorem~\ref{thm:3.1} essentially asserts that it suffices to prove there is no local map into \textit{large} surgery along $K$. We explain how to pass from large surgery to small surgery in Section~\ref{sec:3.2}; this uses a topological argument from \cite[Lemma 4.1]{dai20222}.

The advantage of using large surgeries is that the action of the surgered diffeomorphism $\phi$ is easily computed from the action of $\phi$ on $\CFK(K)$. As is well known, there is a large surgery isomorphism between $\CFm(S^3_n(K), [0])$ and the $A_0$-subcomplex of $\CFK(K)$ for $n \geq g_3(K)$. We verify that this intertwines the action of $\phi$ on the former with the action of $\phi$ on the latter. This is similar to the equivariant large surgery formula from \cite{mallick2022knot}, although due to the fact that $\phi$ is a general symmetry, the proof is not quite the same. 

\subsection{Large surgeries}\label{sec:3.1}
We begin by reviewing a particular formulation of the large surgery isomorphism. Let $K$ be any knot in $S^3$ and let $W_n(K)$ denote the 2-handle cobordism from $S^3$ to $S^3_{n}(K)$. Let $W'_n(K)$ denote the cobordism from $S^3_{n}(K)$ to $S^3$ obtained by turning $W_{n}(K)$ around and swicthing orientation. Puncturing the core of the $2$-handle in $W_n(K)$ gives a cobordism from the unknot $U \subset S^3_{n}(K)$ to $K \subset S^3$ inside $W'_{n}(K)$; see Figure~\ref{fig:3.1}. Denote this by $\Sigma_K$. Decorate $\Sigma_K$ with two arcs running from $U$ to $K$ which separate $\Sigma_K$ into $z$-basepointed and $w$-basepointed regions. Let $\mathcal{F}$ denote $\Sigma_K$ with this decoration; we also consider the conjugate decoration $\overline{\mathcal{F}}$ obtained by switching the $w$ and $z$-regions. In this case, the basepoints on $U$ and $K$ are also switched.

\begin{figure}[h!]
\center
\includegraphics[scale=0.9]{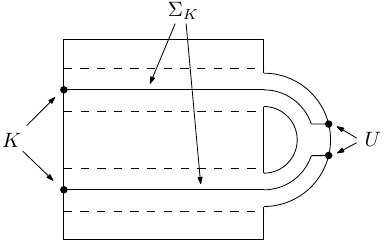}
\caption{The cobordism $W_n(K)$ obtained by attaching a $2$-handle to the outgoing end of $S^3 \times I$, together with the knot cobordism $\Sigma_K$ between $K$ and $U$. In the case that $K$ is equipped with a relative diffeomorphism $\phi$, the dotted lines denote $N(K) \times I$, where $N(K)$ is a neighborhood of $K$ fixed by $\phi$.}\label{fig:3.1}
\end{figure}

Let $\mathfrak{x}$ and $\mathfrak{y}$ be two $\spinc$-structures on $W'_{n}(K)$ such that
\[
\langle c_1(\mathfrak{x}), [\widehat{\Sigma}_K] \rangle = -n \quad \text{and} \quad \langle c_1(\mathfrak{y}), [\widehat{\Sigma}_K] \rangle = n,
\]
where $\widehat{\Sigma}_K$ represents the surface obtained by capping off $\Sigma_K$ by a Seifert surface for $K$ (and closing up the unknot on the other side). It follows that both $\mathfrak{y}$ and $\mathfrak{x}$ restrict to the $\spinc$-structure $[0] \in \Spinc(S^3_{n}(K))$. Note that $\mathfrak{x}$ and $\mathfrak{y}$ are conjugate to each other. 

For $n \geq g_3(K)$, the large surgery isomorphism is realized by the knot Floer cobordism map
\begin{equation}\label{eq:3.1}
F_{W,\mathcal{F},\mathfrak{x}}: \mathcal{CFK}(S^3_{n}(K),U) \rightarrow \mathcal{CFK}(S^3, K).
\end{equation}
By this, we mean the following: the map $F_{W,\mathcal{F},\mathfrak{x}}$ preserves the Alexander grading and hence restricts to a map from the $A_0$-complex of the left-hand side to the $A_0$-complex of the right-hand side. The former is tautologically identified with $\CFm(S^3_n(K), [0])$, while the latter is $A_0(K)$. In \cite[Section~4]{OSKnots} \cite{RasmussenThesis}, it is shown that this restriction is an isomorphism of $\F[U]$-complexes. See also \cite[Proposition~6.9]{HM}. Note that the surgered diffeomorphism $\phi$ of $S^3_n(K)$ fixes $U$ pointwise and hence induces a self-map of $\CFK(S^3_n(K), U)$, which we again denote by $\phi$.

\begin{lemma}\label{lem:3.2}
Let $\phi$ be a relative diffeomorphism of $(S^3, K)$. Then the map
\[
F_{W,\mathcal{F},\mathfrak{x}}: \mathcal{CFK}(S^3_{n}(K),U) \rightarrow \mathcal{CFK}(S^3, K)
\]
homotopy commutes with both $\phi$ and $\iota_K$.
\end{lemma}
\begin{proof}
We first re-phrase the proof of \cite[Theorem~1.5]{HM} to verify the commutation relation
\[
F_{W,\mathcal{F},\mathfrak{x}} \circ \iota_U \simeq \iota_K \circ F_{W,\mathcal{F},\mathfrak{x}}.
\]
It follows from \cite[Theorem 1.3]{Zemkeconnected} that
\[
F_{W,\mathcal{F},\mathfrak{x}} \circ \iota_U \simeq \iota_K \circ F_{W,\mathcal{F},\overline{\mathfrak{x}} + \mathrm{PD}[\Sigma_K]}.
\]
Note that $\mathfrak{x}$ and $\mathfrak{y}$ are defined from the $\spinc$-equivalence class with respect to the basepoints $w$ and $z$ respectively \cite{OSKnots}. In particular, we have $\mathfrak{x} - \mathfrak{y} = \mathrm{PD}[\Sigma_K]$, which proves the claim.

It remains to show that $F_{W,\mathcal{F},\mathfrak{x}}$ homotopy commutes with $\phi$. This is straightforward: note that $\phi$ extends over $W'_n(K)$ as $\phi \times \id$, together with the identity on the $2$-handle attachment. This extension fixes $\Sigma_K$ pointwise and is easily checked to act as the identity on the set of $\spinc$-structures on $W'_n(K)$. It follows that 
\[
\phi \circ F_{W,\mathcal{F},\mathfrak{x}} \simeq  F_{W,\mathcal{F},\mathfrak{x}}\circ \phi
\]
by the diffeomorphism invariance of the link cobordism maps; see \cite[Theorem~A]{zemke2019link} and \cite[Equation~(1.2)]{Zemkegraphcobord}.
\end{proof} 

This immediately gives:

\begin{lemma}\label{lem:3.3}
Let $\phi$ be a relative diffeomorphism of $(S^3, K)$. For $n \geq g_3(K)$, we have a homotopy equivalence
\[
(\CFm(S^3_n(K), [0]), \phi, \iota) \simeq (A_0(K), \phi, \iota_K).
\]
\end{lemma}
\begin{proof}
As stated previously,
\[
F_{W,\mathcal{F},\mathfrak{r}}: \mathcal{CFK}(S^3_{n}(K),U) \rightarrow \mathcal{CFK}(S^3, K)
\]
induces an isomorphism between the $A_0$-complex of the left-hand side and the $A_0$-complex of the right-hand side. The former is tautologically identified with $\CFm(S^3_n(K), [0])$; this identification takes $\iota_U$ to $\iota$ and the action of $\phi$ on $\mathcal{CFK}(S^3_{n}(K),U)$ to the action of $\phi$ on $\CFm(S^3_n(K), [0])$. Applying Lemma~\ref{lem:3.2} then gives the claim.
\end{proof}

\subsection{Small surgeries}\label{sec:3.2}
We now explain how to pass from large to small surgery. In what follows, our convention is that $L(m,1)$ is $m$-surgery on the unknot. 

\begin{lemma}\label{lem:3.4}
Let $K$ be any knot and $m$ and $n$ be any two positive integers.
\begin{enumerate}
\item There is a negative-definite cobordism $W_1$ from 
\[
S^3_{+1}(K) \quad \text{to} \quad S^3_{n}(K)
\]
with $b_1(W_1) = 0$. This cobordism is spin.
\item There is a negative-definite cobordism $W_2$ from 
\[
S^3_{1/(m+1)}(K) \quad \text{to} \quad S^3_{+1}(K)\# L(-m,1)
\]
with $b_1(W_2) = 0$. This cobordism is spin if and only if $m$ is even.
\end{enumerate}
\end{lemma}
\begin{proof}
The cobordism $W_1$ is given by attaching $n-1$ meridional 2-handles to $S^3_{+1}(K)$, as displayed in Figure~\ref{fig:3.3}. 

\begin{figure}[h!]
\center
\includegraphics[scale=0.8]{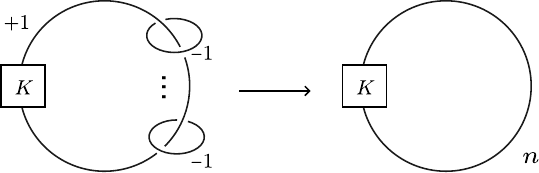}
\caption{A cobordism from $S^3_{+1}(K)$ to $S^3_n(K)$ given by attaching $n-1$ meridional $2$-handles along $(-1)$-framed meridians of $K$.}
\label{fig:3.3}
\end{figure}

\noindent
The linking form of the $n$-component link on the left is:
\vspace{0.1cm}
\[
\left(\begin{array}{ccccc}1 & 1 & 1 &  & 1 \\1 & -1 & 0 & \cdots & 0 \\1 & 0 & -1 &  & 0 \\ & \vdots &  & \ddots &  \\1 & 0 & 0 &  & -1\end{array}\right).
\vspace{0.2cm}
\]

\noindent
The second homology of this cobordism is given by the orthogonal complement of the first column, which has a basis given by $\{(1, -1, 0, \ldots, 0), (1, 0, -1, \ldots, 0), \ldots, (1, 0, 0, \ldots, -1)\}$. Each of these has self-intersection $-2$, while each pair of distinct basis elements has intersection $-1$. It follows that $W_1$ is negative-definite and spin. See \cite[Lemma 4.1]{dai20222}.

Figure~\ref{fig:3.4} displays a cobordism from $S^3_{+1}(K)\# L(-m,1)$ to $\smash{S^3_{1/(m+1)}(K)}$, obtained by attaching a single $2$-handle.
 
\begin{figure}[h!]
\center
\includegraphics[scale=0.8]{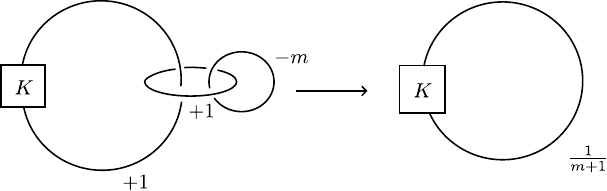}
\caption{A cobordism from $S^3_{+1}(K)\# L(-m,1)$ to $\smash{S^3_{1/(m+1)}(K)}$ given by attaching a $2$-handle along the $(+1)$-framed curve that links $K$.}
\label{fig:3.4}
\end{figure}

\noindent
To calculate the intersection form of this cobordism, observe that the linking form of the $3$-component link on the left is
\vspace{0.1cm}
\[
\begin{pmatrix}
1& 1& 0\\
1& 1& 1\\
0& 1&-m
\end{pmatrix}.
\vspace{0.2cm}
\]
\noindent
The second homology of this cobordism is given by the orthogonal complement of the first and third columns, which is spanned by $(m,-m,-1)$. This has self-intersection $m$. The cobordism $W_2$ is obtained by turning the cobordism of Figure~\ref{fig:3.4} around.
\end{proof}

Now suppose $\phi$ is a relative diffeomorphism of $(S^3, K)$. Then $W_1$ and $W_2$ are equivariant with respect to placing the surgered diffeomorphism $\phi$ on both ends, as can be seen by putting the handle attachment regions of Figures~\ref{fig:3.3} and \ref{fig:3.4} sufficiently close to $K$. Here, we define $\phi$ on $\smash{S^3_{+1}(K)\# L(-m,1)}$ by placing the connected sum point near $K$, so that $\phi$ extends to a self-diffeomorphism of $\smash{S^3_{+1}(K)\# L(-m,1)}$ which is the identity on the second summand. It is straightforward to check that in each case, the extension over the cobordism fixes the second homology and hence the set of $\spinc$-structures. This gives:

\begin{lemma}\label{lem:3.5}
Let $K$ be any knot and $m$ be positive and even. Fix any relative diffeomorphism $\phi$ of $(S^3, K)$. Then there are local maps
\[
F_1 \colon (\CFm(S^3_{+1}(K)), \phi, \iota) \rightarrow (A_0(K), \phi, \iota_K)
\]
and
\[
F_2 \colon (\CFm(S^3_{1/(m+1)}(K)), \phi, \iota) \rightarrow (\CFm(S^3_{+1}(K)), \phi, \iota)
\]
\end{lemma}

\begin{proof}
To define $F_1$, let $n \geq g_3(K)$ be odd and $W_1$ be the cobordism from Lemma~\ref{lem:3.4}. Denote the unique self-conjugate $\spinc$-structure on $W_1$ by $\s_0$. Then Lemma~\ref{lem:3.4} combined with Lemma~\ref{lem:2.5} gives a local map
\[
F_{W_1, \s_0} \colon (\CFm(S^3_{+1}(K)), \phi, \iota) \rightarrow (\CFm(S^3_{n}(K), [0]), \phi, \iota).
\]
of grading shift $(n-1)/4$. We now invoke the equivalence
\[
(\CFm(S^3_n(K), [0]), \phi, \iota) \simeq (A_0(K), \phi, \iota_K),
\]
of Lemma~\ref{lem:3.3}, which has grading shift $-(n-1)/4$, see \cite[Section 4]{OSKnots}. Postcomposing $F_{W_1, \s_0}$ with this identification gives the desired map $F_1$.

The map $F_2$ is slightly more subtle. Consider the cobordism $W_2$ constructed in Lemma~\ref{lem:3.4}. Denote the unique self-conjugate $\spinc$-structure on $W_2$ by $\s_0$. We claim that $\s_0$ restricts to the self-conjugate $\spinc$-structure on $L(-m,1)$ which corresponds to $[m/2]$. 

To see this, consider the cobordism $W$ from $S^3_{+1}(K)$ to $S_{1/(m+1)}(K)$ obtained by attaching a $(+1)$-framed $2$-handle along an unknot $U$ that links $K$ once and a $(-m)$-framed $2$-handle attached along another unknot $U'$ that links $U$ once, as in Figure~\ref{fig:3.4}. Let $A$ be the subcobordism from $S^3_{+1}(K)$ to $S^3_{+1}(K) \# L(-m, 1)$ obtained from the handle attachment along $U'$; this is just the cylinder $S^3_{+1}(K) \times I$ boundary sum the usual lens space cobordism $W_L$ from the empty set to $L(-m, 1)$. Let $B$ be the subcobordism from $S^3_{+1}(K) \# L(-m, 1)$ to $S^3_{1/(m+1)}(K)$ obtained from the handle attachment along $U$; this is just $-W_2$. Then $W = A \cup B$ and it is straightforward to check that $W$, $A$, and $B$ each have even intersection form (in the first two cases by sliding $U$ over $K$). Hence each has a unique self-conjugate $\spinc$-structure, and the unique self-conjugate $\spinc$-structure on $W$ moreover restricts to the unique self-conjugate $\spinc$-structures on $A$ and $B$, the latter of which is $s_0$. It follows that $s_0$ restricts to a self-conjugate $\spinc$-structure on $L(-m, 1)$ that extends over $W_L$. This is the characterizing property of $[m/2]$.

Lemma~\ref{lem:3.4} combined with Lemma~\ref{lem:2.5} now gives a local map
\[
F_{W_2, \s_0} \colon (\CFm(S^3_{1/m}(K)), \phi, \iota) \rightarrow (\CFm(S^3_{+1}(K) \# L(-m, 1), [m/2]), \phi, \iota).
\]
of grading shift $1/4$. By the usual connected sum formula,
\begin{equation}\label{eq:3.3}
\CFm(S^3_{+1}(K) \# L(-m, 1)) \simeq \CFm(S^3_{+1}(K)) \otimes \CFm(L(-m, 1)).
\end{equation}
As shown in \cite[Theorem 1.1]{HMZ}, \eqref{eq:3.3} intertwines the $\iota$-action on the left with the tensor product $\iota$-action $\iota \otimes \id$ on the right. It is also straightforward to see \eqref{eq:3.3} intertwines $\phi$ on the left with $\phi \otimes \id$ on the right. (See Lemma~\ref{lem:4.2} below.) Postcomposing $F_{W_2, \s_0}$ with \eqref{eq:3.3} and using the fact that $L(-m, 1)$ is an L-space with $d(L(-m, 1), [m/2]) = -1/4$ gives the desired map $F_2$.
\end{proof}

Everything is now in place to prove Theorem~\ref{thm:3.1}:

\begin{proof}[Proof of Theorem~\ref{thm:3.1}]
Let $m$ be positive and odd. Suppose that $\smash{(S^3_{1/m}(K), \phi)}$ bounded a homology ball $W_0$ with an extension of $\phi$. Then Lemma~\ref{lem:2.5} would give a local map
\[
F_{W_0} \colon (\F[U], \id, \id) \rightarrow (\CFm(S^3_{1/m}(K)), \phi, \iota).
\]
Postcomposing this with $F_1$ (if $m = 1$) or $F_1 \circ F_2$ (if $m > 1$) from Lemma~\ref{lem:3.5} then gives a local map from $(\F[U], \id, \id)$ to $(A_0(K), \phi, \iota_K)$. This contradicts the hypotheses of the theorem.
\end{proof}


\section{Split Diffeomorphisms}\label{sec:4}

We now consider the case where $\phi = \phi_1 \# \phi_2$ is a split diffeomorphism. In this setting, we have the following algebraic re-interpretation of Theorem~\ref{thm:3.1}:

\begin{theorem}\label{thm:4.1}
Let $\phi = \phi_1 \# \phi_2$ be a split diffeomorphism of a slice knot $K = K_1 \# K_2$. Suppose that there is no local map
\[
(\CFK(K_2), \phi_2, \iota_{K_2})^\vee \rightarrow (\CFK(K_1), \phi_1, \iota_{K_1}).
\]
Then $(S^3_{1/m}(K), \phi)$ is a strong cork for any $m$ positive and odd.
\end{theorem}

\subsection{Proof of Theorem~\ref{thm:4.1}}\label{sec:4.1} Let $\phi = \phi_1 \# \phi_2$ be a split diffeomorphism of $K_1 \# K_2$. Recall that we have a homotopy equivalence
\[
h:\mathcal{CFK}(K_1 \# K_2) \rightarrow \mathcal{CFK}(K_1)\otimes \mathcal{CFK}(K_2).
\]
This was first shown in \cite[Theorem 7.1]{OSKnots} and later re-interpreted in terms of an explicit pair-of-pants cobordism in \cite[Proposition 5.1]{Zemkeconnected}. We begin by computing the action of $\phi$ on $\CFK(K_1 \# K_2)$ under this identification.

\begin{lemma}\label{lem:4.2}
Let $\phi = \phi_1 \# \phi_2$ be a split diffeomorphism of $K = K_1 \# K_2$. Then we have a homotopy equivalence
\[
(\mathcal{CFK}(K_1 \# K_2), \phi_1 \# \phi_2, \iota_{K_1 \# K_2}) \simeq (\CFK(K_1) \otimes \CFK(K_2), \phi_1 \otimes \phi_2, \iota_\otimes),
\] 
where $\iota_{\otimes} = (\id \otimes \id + \Phi \otimes \Psi) \circ (\iota_{K_1} \otimes \iota_{K_2})$.
\end{lemma}

\begin{proof}
This was essentially shown in \cite[Theorem~5.1]{JuhaszZemkeSliceDisks}. The homotopy equivalence $h$ is given by the link cobordism map $F_{W,\cF}$, where $\cF$ is the cobordism built by attaching a fission band which splits $K_1 \# K_2$ into $K_1 \sqcup K_2$ and $W$ is built by attaching a 3-handle which splits $(S^3,K_1\sqcup K_2)$ into $(S^3, K_1)\sqcup (S^3,K_2)$. It was shown in \cite[Theorem 1.1]{Zemkeconnected} that $h$ intertwines $\iota_{K_1 \# K_2}$ and $\iota_\otimes$.

It is clear that both the fission band and the attaching sphere of the 3-handle can be chosen to be fixed by $\phi$. It is thus easily checked that $\phi$ extends over the cobordism $W$ in such a way that the extension fixes $\cF$ pointwise. On the outgoing component $(S^3, K_1)$ this extension acts as $\phi_1$, while on the outgoing component $(S^3, K_2)$ this extension acts as $\phi_2$. The theorem thus follows immediately from diffeomorphism invariance of the link cobordism maps; see \cite[Theorem~A]{zemke2019link} and \cite[Equation~(1.2)]{Zemkegraphcobord}. 
\end{proof}

We now turn to the proof of Theorem~\ref{thm:4.1}. We first have:

\begin{lemma}
\label{lem:C_1<C_2}
 Let $\cC_1$ and $\cC_2$ be two $(\phi,\iota_K)$-complexes. There is a local map from $(\bF[\scU,\scV], \mathrm{id}, \iota_0)$ to  $\cC_1\otimes \cC_2^\vee$ if and only if there is a local map from $\cC_2$ to $\cC_1$.
\end{lemma}
\begin{proof} This follows immediately from the group structure on the set of local classes. If there is a local map from $\cC_2$ to $\cC_1$, then we can tensor with the identity map on $\cC_2^{\vee}$ to get a local map from $\cC_2\otimes \cC_2^{\vee}$ to $\cC_1\otimes \cC_2^{\vee}$. There is always a local map from $(\bF[\scU,\scV], \id, \iota_0)$ to $\cC_2\otimes \cC_2^{\vee}$, so composing we get a local map from $(\bF[\scU,\scV], \id, \iota_0)$ to $\cC_1\otimes \cC_2^{\vee}$. The converse is similarly straightforward to establish.
\end{proof}

The following is also useful for our purposes:
\begin{lemma}\label{lem:fromA_0toC} Let $\cC=(C,\phi, \iota_K)$ be a $(\phi,\iota_K)$-class. Then there is a local map from $(\bF[\scU,\scV],\mathrm{id}, \iota_0)$ to $\cC$ if and only if there is a local map from $(\bF[U], \id,\id)$ to $(A_0(C),\phi, \iota_K)$.
\end{lemma}
\begin{proof} The ``only-if'' direction is obvious. Conversely, a local map from $(\bF[U], \id,\id)$ to $A_0(\cC)$ consists of an element $x\in A_0(C)$ such that $\iota_K(x)+x=\d(z)$ and $\phi(x)+x=\d(y)$ for some $z,y\in A_0(C)$. Since we can view $x$, $y$ and $z$ as also being elements of $\cC$, this is the exact same data as a local map from $(\bF[\scU,\scV],\id, \iota_0)$ to $(\cC,\phi, \iota_K)$. 
\end{proof}

As a consequence of Lemmas~\ref{lem:C_1<C_2} and~\ref{lem:fromA_0toC}, we immediately obtain the following:

\begin{corollary}\label{cor:local-equivalence-restatement} Let $\cC_1$ and $\cC_2$ be two $(\phi, \iota_K)$-complexes. There is a local map $(\bF[U], \id, \id) \rightarrow A_0(\cC_1 \otimes \cC_2^\vee)$ if and only if there is a local map from $\cC_2$ to $\cC_1$.
\end{corollary}
\begin{proof}
Follows from Lemmas~\ref{lem:C_1<C_2} and~\ref{lem:fromA_0toC}.
\end{proof}

The proof of Theorem~\ref{thm:4.1} is now clear:

\begin{proof}[Proof of Theorem~\ref{thm:4.1}]
Let $m$ be positive and odd. Suppose that $\smash{(S^3_{1/m}(K_1 \# K_2), \phi_1 \# \phi_2)}$ bounded a homology ball $W_0$ with an extension of $\phi_1 \# \phi_2$. By Theorem~\ref{thm:3.1}, there is a local map
\[
(\F[U], \id, \id) \rightarrow (A_0(K_1 \# K_2), \phi_1 \# \phi_2, \iota_{K_1 \# K_2}).
\]
Setting $\cC_1 = (\CFK(K_1), \phi_1, \iota_{K_1})$ and $\cC_2 = (\CFK(K_2), \phi_2, \iota_{K_2})$, Corollary~\ref{cor:local-equivalence-restatement} shows there is a local map from $\cC_2^\vee$ to $\cC_1$, as desired.
\end{proof}

\subsection{The swallow-follow diffeomorphism}\label{sec:4.2}
We now finally specialize to the case when our split diffeomorphism $\phi_1 \# \phi_2$ is $\smash{(\tl^i \tm^j) \# \id}$. We then use the action on knot Floer homology to calculate the action of $\tl$ (and $\tm$) on large surgeries along $K \# -K$. As we will see, it will be necessary to simultaneously keep track of the Heegaard Floer involution $\iota$. Although straightforward, we record the calculation below:

\begin{theorem}\label{thm:1.5}
Let $K_1$ and $K_2$ be any pair of knots and $n \geq g_3(K_1 \# K_2)$. Then we have a homotopy equivalence of tuples
\[
(\CFm(S^3_n(K_1 \# K_2), [0]), \tl, \tm, \iota) \simeq (A_0(K_1 \# K_2), s \otimes \id, \id, \iota_\otimes)
\]
where 
\[
s=\id+\Phi\Psi \quad \text{and} \quad \iota_{\otimes} = (\id \otimes \id + \Phi \otimes \Psi) \circ (\iota_{K_1} \otimes \iota_{K_2}).
\]
\end{theorem}

\begin{proof}[Proof of Theorem~\ref{thm:1.5}]
This follows immediately from Lemma~\ref{lem:4.2}, which gives a homotopy equivalence
\[
(\CFK(K_1 \# K_2), \tl \# \id, \tm \# \id, \iota_{K_1 \# K_2}) \simeq (\CFK(K_1) \otimes \CFK(K_2), \tl \otimes \id, \tm \otimes \id, \iota_\otimes).
\]
As is well-known, the longitudinal twist $\tl$ on $K_1$ acts as the Sarkar map $s$ on $\CFK(K_1)$ \cite{sarkar2015moving, Zemkequasistab}. The meridional twist $\tm$ on $K_1$ acts as the identity, since it is isotopic to the identity through an isotopy which fixes $K_1$ pointwise. The claim then follows from Lemma~\ref{lem:3.3}.
\end{proof}

We now complete the proof of the main theorem:

\begin{proof}[Proof of Theorem~\ref{thm:1.2}]
Suppose that we had an extension of $(\smash{\tl^i \tm^j}) \# \id$ over $\smash{S^3_{1/m}(K \# -K)}$ for some $m$ positive and odd. As in the proof of Theorem~\ref{thm:1.5}, we know that $\tl$ acts on $\CFK(K_1)$ by the Sarkar map $s$, while $\tm$ acts on $\CFK(K_1)$ by the identity. Using Theorem~\ref{thm:4.1} together with the fact that $s^2 \simeq \id$, we obtain a local map
\[
f \colon (\CFK(K), \id, \iota_K) \rightarrow (\CFK(K), \tl^i\tm^j, \iota_K) \simeq (\CFK(K), s, \iota_K).
\]
This may be thought of as a self-local map of the $\iota_K$-complex $\cC = (\CFK(K), \iota_K)$ satisfying the additional condition
\[
s \circ f \simeq f.
\]
Now let $\cC_{\mathrm{conn}}$ be the connected complex of $\cC$, so that there are $(\iota_K$-)local maps $h_1 \colon \cC_{\mathrm{conn}} \rightarrow \cC$ and $h_2 \colon \cC \rightarrow \cC_{\mathrm{conn}}$. Then $f$ induces a self-local map of $\cC_{\mathrm{conn}}$ given by $\tilde{f} = h_2 \circ f \circ h_1$. Since $s$ commutes with any chain map up to homotopy (using \cite[Lemma 2.8]{Zemkeconnected} and the fact that  $s= \mathrm{id} + \Phi \Psi$), we again have
\[
s \circ \tilde{f} \simeq \tilde{f}.
\]
On the other hand, by Lemma~\ref{lem:2.10} we know that $\tilde{f}$ is a chain isomorphism. It follows that $s \simeq \id$ on $\cC_{\mathrm{conn}}$, as desired. The case when $m$ is negative and odd follows from a similar argument after mirroring and reversing orientation. (It is easily checked that $\cC$ is $\Ds$-nontrivial if and only if $\cC^\vee$ is $\Ds$-nontrivial.)
\end{proof}


\section{Examples and further discussion}\label{sec:5}
In order to demonstrate the broad applicability of our obstruction, we now give several examples of Theorems~\ref{thm:1.2} and \ref{thm:numerical-intro}. 

\subsection{Further examples of Gompf's construction}\label{sec:5.1}
We begin with the proof of Corollary~\ref{cor:1.3}:

\begin{proof}[Proof of Corollary~\ref{cor:1.3}]
We claim that if a knot $K$ satisfies the condition from the hypothesis, that is:
\begin{equation}\label{hypothesis}
    2 \Arf(K) + |\tau(K)| \equiv 1 \text{ or } 2 \bmod 4,
\end{equation}    
then the connected complex $C_{\mathrm{conn}}(K)$ consists of a step-length-one staircase (or possibly a single dot) together with a side-length-one box, as schematically shown in Figure~\ref{fig:figure_eight_sarkar}. Note that it follows from \cite[Proposition 8.2]{HM} that if the number of boxes in the main diagonal is odd, then the connected complex $C_{\mathrm{conn}}(K)$ has the form mentioned above.
\begin{figure}[h!]
\center
\includegraphics[scale=0.8]{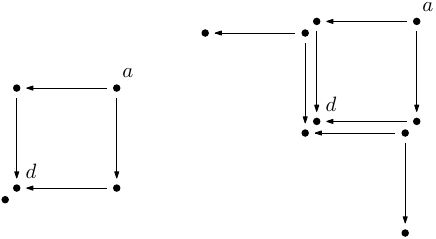}
\caption{Left: the connected complex for $K = 4_1$. Right: the connected complex for $T_{2,3} \# T_{2,3}$. In both cases, $s(a) = a + d$. We have omitted writing $\cU$ and $\cV$ multiples of the generators for brevity.}
\label{fig:figure_eight_sarkar}
\end{figure}

Moreover, it follows from the structure of the knot Floer complex of a thin knot \cite{petkovacables} that the parity of the following expression determines the parity of the number of boxes in the main diagonal
\begin{equation}\label{determinant}
    \frac{D - 2|\tau| -1}{4}.
\end{equation}
See for example \cite[Section 8.1]{HM}, here $D$ is the determinant of a knot. Hence to prove our claim, it suffices to show that (\ref{hypothesis}) implies that (\ref{determinant}) is odd. Note that $\Arf(K) = 0$ if and only if $D \equiv \pm 1 \bmod 8$, which translates to the following relation:
\begin{equation}\label{arf}
    D + 4 \Arf(K) \equiv \pm 1 \bmod 8.
\end{equation}
Now note that (\ref{determinant}) is odd if and only if
\begin{equation}\label{to_show}
D - 2|\tau| \equiv 5 \bmod 8.
\end{equation}
Replacing $D$ from (\ref{arf}) in the left side of (\ref{to_show}) we get
\begin{equation}
 - 4 \Arf(K) \pm 1 - 2 |\tau|.   
\end{equation}
Using (\ref{hypothesis}), we get that the above expression can only take values in $\{5,3,7 \} \bmod 8$ regardless of the value of $\Arf(K)$. However, if it is either $\{3,7 \} \bmod 8$ then (\ref{determinant}) is not an integer, so  $D - 2|\tau| \equiv 5 \bmod 8$ , as required.

Now it follows that in the connected complex $C_{\mathrm{conn}}(K)$, the Sarkar map $s$ sends $s(a) = a + d$ and is the identity otherwise. It is then straightforward to check that $K$ is $\Ds$-nontrivial since we can rule out the possibility of $s \simeq \mathrm{id}$ on this complex. Hence the result follows from Theorem~\ref{thm:1.2}.
\end{proof}

We now move on to examples where $K$ is a connected sum of torus knots. These are especially interesting due to the results of \cite{Gompfhandle} and \cite{RayRuberman}, where it was shown that for any individual torus knot $K$, the longitudinal and meridional twists extend over $C_{K, m}$ for any $m \in \Z^{\neq 0}$. Hence for $K$ a torus knot, no twist along the swallow-follow torus makes $C_{K, m}$ into a cork. In contrast, we show that for $K$ a connected sum of torus knots, Gompf's construction often yields a strong cork. This includes the simplest case of $K = T_{2, 3} \# T_{2, 3}$:

\begin{corollary}\label{cor:trefoil_connected}
Let $K = s T_{2, 2n + 1}$ for $n$ odd and $s \equiv 2 \text{ or } 3 \bmod 4$. Then  $(Y_{K, m}, \tl^i \tm^j)$ is a strong cork for all $(m, i, j) \in \Z^3$ with $m$ and $i$ both odd.
\end{corollary}
\begin{proof}
Note that $K$ is Floer-thin since $T_{2, 2n+1}$ is alternating. It is a standard fact that
\[
\Arf(T_{2, 2n+1}) = 
\begin{cases}
1 &\text{ if } n \equiv 1 \text{ or } 2 \bmod 4 \\
0 &\text{ if } n \equiv 0 \text{ or } 3 \bmod 4
\end{cases}
\quad \text{and} \quad \tau(T_{2, 2n+1}) = n.
\]
The additivity of $\Arf$ and $\tau$ then gives $\Arf(K)$ and $\tau(K)$. Exhaustive casework then shows that the hypotheses of Corollary~\ref{cor:1.3} hold precisely when $n$ is odd and $s \equiv 2 \text{ or } 3 \bmod 4$. 
\end{proof}

Corollaries~\ref{cor:1.3} and \ref{cor:trefoil_connected} deal with thin knots, which are some of the simplest knots from the point of view of knot Floer homology. Note that all of the knots discussed in \cite{gompf2017infinite} are thin. However, our obstruction is certainly capable of producing strong corks from non-thin knots:

\begin{corollary}\label{cor:nonthin}
Let $K = -2T_{2n, 2n+1} \# T_{2n, 4n+1}$ for $n$ odd. Then $(Y_{K, m}, \tl^i \tm^j)$ is a strong cork for all $(m, i, j) \in \Z^3$ with $m$ and $i$ both odd.
\end{corollary}
\begin{proof}
For $n$ odd, the connected complex of $-2T_{2n, 2n+1} \# T_{2n, 4n+1}$ was computed in \cite{hendricks2022quotient}. The result is the same as shown on the left in Figure~\ref{fig:figure_eight_sarkar}, except that the lengths of the arrows appearing in the differential are larger and odd. It is easily checked that  $s \nsimeq \mathrm{id}$ on $C_{\mathrm{conn}}(K)$.
\end{proof}

Many similar examples are possible using linear combinations of L-space knots; we present Corollary~\ref{cor:nonthin} due to the fact that the requisite computation already appears in the literature \cite{hendricks2022quotient}.

\begin{remark}
The preceding examples have primarily focused on the swallow-follow operation on $Y_{K, m}$ due to its connection to \cite{gompf2017infinite}. As discussed in Section~\ref{sec:2.2}, however, we can instead let $K = K_1$ and $K_2$ be any concordance inverse to $K_1$. It is not difficult to check that each of the instances of $Y_{K, m}$ in this paper can generalized to $\smash{S^3_{1/m}(K_1 \# K_2)}$.
\end{remark}

\subsection{More general diffeomorphisms}\label{sec:5.2}
We now move on to more general examples of corks where the underlying diffeomorphism is not a twist along the swallow-follow torus. A wide range of such examples come from periodic involutions on knots. Recall that a knot $K$ in $S^3$ is 2-\textit{periodic} if there exists an orientation-preserving involution $\tau$ of $S^3$ that preserves the oriented knot $K$ setwise. The action of such an involution on knot Floer homology was considered in \cite{dai2023equivariant, mallick2022knot}. 

By postcomposing $\tau$ with a half-Dehn twist along $K$, we obtain a relative diffeomorphism of $(S^3, K)$ which by abuse of notation we also denote by $\tau$.\footnote{In fact, since knot Floer homology is a doubly-basepointed theory, this composition is necessary in order to define the action of $\tau$ on $\CFK(K)$. Technically, we must also make sure to perform the half-Dehn twist along the orientation of $K$. See \cite[Section 2.2]{mallick2022knot} for details.} Note that as a relative diffeomorphism, $\tau^2$ is isotopic to the Sarkar basepoint-pushing map on $K$, see \cite[Proposition 2.6]{mallick2022knot}. In the case that $K$ is a $2$-periodic knot, we may thus think of $\tau$ as forming a square root of $s$. We have the following simple example:

\begin{corollary}\label{cor:periodic}
Let $(K,\tau)$ be a 2-periodic knot. If $K$ is $\Ds$-nontrivial, then $(Y_{K,m}, (\tau \# \mathrm{id})^{i})$ is a strong cork for any $m$ odd and $i \not\equiv 0 \bmod 4$.
\end{corollary}
\begin{proof}
As discussed previously, we have $\tau^2 \simeq s$. The $\Ds$-nontriviality of $K$ thus implies that $(\tau \# \id)^2$ makes $Y_{K,m}$ into a strong cork. It follows that $\tau \# \id$ makes $Y_{K,m}$ into a strong cork, since if $\tau \# \id$ extended over some homology ball, so too would $(\tau \# \id)^2$. A similar observation for $(\tau \# \id)^6 \simeq s^3 \# \id$ likewise shows the same for $(\tau \# \id)^3$. Noting that $\tau^4 \simeq s^2 \simeq \mathrm{id}$ as self-maps of $\CFK(K)$ easily gives the claim.
\end{proof}
There are many examples of 2-periodic knots that are also $\Ds$-nontrivial. For instance, Figure~\ref{fig:figure_eight} shows that the two simplest $\Ds$-nontrivial knots $K = 4_1$ and $5_2$ admit such a periodic involution.

\begin{figure}[h!]
\center
\includegraphics[scale=0.18]{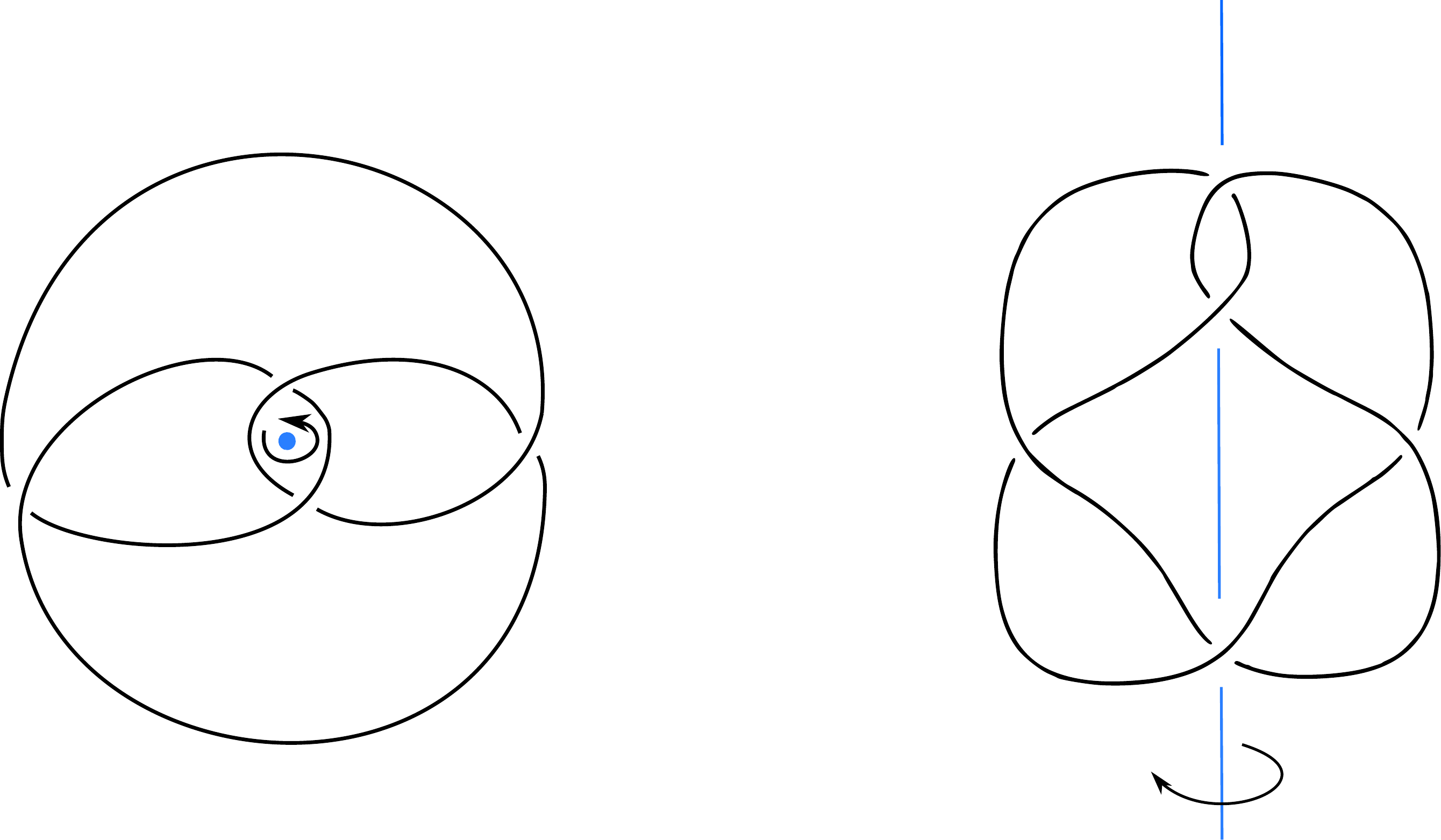}
\caption{The knots $4_1$ and $5_2$ with the periodic involutions $\tau$.}
\label{fig:figure_eight}
\end{figure}

We now give an example which is not based on $\Ds$-nontriviality and instead utilizes Theorem~\ref{thm:numerical-intro}. Let $K_1$ and $K_2$ be slice and consider a split diffeomorphism on $K_1 \# K_2$ of the form $\phi \# \id$. Since $K_2$ is slice, $(\CFK(K_2), \id, \iota_{K_2})$ is locally trivial. It easily follows that 
\[
\delta(K_1 \# K_2, \phi \# \id) = \delta(K_1, \phi).
\]
We use this to give an example of a cork with a slightly more subtle boundary diffeomorphism:

\begin{corollary}\label{cor:connect_periodic}
Let $K_1 = 4_1 \# 4_1$ and equip $K_1$ with the split diffeomorphism $\phi = \tau \# \tau$. Let $K_2$ be any slice knot. Then
\[
(S^3_{1/m}(K_1 \# K_2), \phi \# \id)
\]
is a strong cork for any $m$ positive and odd.
\end{corollary}
\begin{proof}
By Theorem~\ref{thm:numerical-intro}, it suffices to show that 
\[
\delta(4_1 \# 4_1, \tau \# \tau) = \delta(K_1, \phi) = \delta(K_1 \# K_2, \phi \# \id) > 0. 
\]
We check this by showing that there is no $U$-nontorsion homology class in $H_{*}(A_0(K_1))$ which lies in grading zero and is fixed by both the action of $\tau \# \tau$ and $\iota$. The desired result then follows from the definition of the $\delta$-invariant. 

Label the generators of $\CFK(4_1)$ as in Figure~\ref{fig:figure_eight_actions}. The cycles in $A_0(4_1 \# 4_1)$ are as in Table~\ref{table:computation}, where all but $x|x$ are $U$-torsion. The action of $\tau$ on $4_1$ follows from \cite[Theorem 1.7]{mallick2022knot}; Lemma~\ref{lem:4.2} then gives the computation in Table~\ref{table:computation}. A similar calculation appears in \cite[Section 3.2]{dai20222}.
 
\begin{figure}[h!]
\center
\includegraphics[scale=0.8]{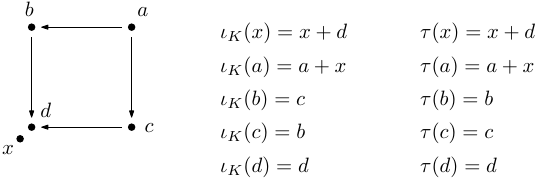}
\caption{The $\CFK(4_1)$ with the action of $\tau$ and $\iota_K$.}
\label{fig:figure_eight_actions}
\end{figure}

\begin{table}[h!]
\begin{tabular}{|c|c|c|c|}
 \hline
 Generators of homology & Image under $\iota$ & Image under $\tau | \tau$ \\
 \hline
 $x|x$ & $x|x+x|d+d|x+d|d$ & $x|x+x|d+d|x+d|d$ \\
 $x|d$ & $x|d+d|d$ & $x|d+d|d$ \\
 $d|x$  & $d|x+d|d$ & $d|x+d|d$ \\
 $a|d+d|a+b|c+c|b$ & $a|d+d|a+b|c+c|b+x|d+d|x+d|d$ & $a|d+d|a+b|c+c|b+x|d+d|x$ \\
 $d|d$ & $d|d$ & $d|d$ \\
 \hline
\end{tabular}
\captionsetup{justification=centering}
\caption{Actions of $\iota$ and $\tau | \tau$ on $H_*(A_0(K))$.}\label{table:computation}
\end{table}

As in \cite[Lemma 3.3]{dai20222}, the $\iota$-invariant subspace of $H_*(A_0(4_1 \# 4_1))$ is spanned by
\[
[x|x] + [x|d] + [a|d + d|a + b|b + c|c], \quad [x|d] + [d|x] \quad \text{and} \quad [d|d].
\]
The latter two generators are $\tau | \tau$-invariant, while the first is not. Since only the first generator is $U$-nontorsion, the claim easily follows.
\end{proof}

Unlike Corollary~\ref{cor:periodic}, the above example cannot be obtained by passing to the squared diffeomorphism. Indeed, it can be checked that $s \otimes s$ acts as identity on $H_{*}(A_0(K))$; hence our use of Theorem~\ref{thm:numerical-intro} is essential. 



\begin{thebibliography}{00}

\bibitem{Akbulut}
Selman Akbulut, \emph{A fake compact contractible {$4$}-manifold}, J. Differential Geom. \textbf{33} (1991), no.~2, 335--356. \MR{1094459}

\bibitem{AKS}
Antonio Alfieri, Sungkyung Kang, and Andr\'{a}s~I. Stipsicz, \emph{Connected
  {F}loer homology of covering involutions}, Math. Ann. \textbf{377} (2020),
  no.~3-4, 1427--1452. \MR{4126897}

\bibitem{CFHS}
C.~L. Curtis, M.~H. Freedman, W.~C. Hsiang, and R.~Stong, \emph{A decomposition
  theorem for {$h$}-cobordant smooth simply-connected compact {$4$}-manifolds},
  Invent. Math. \textbf{123} (1996), no.~2, 343--348. \MR{1374205}

\bibitem{DHM}
Irving Dai, Matthew Hedden, and Abhishek Mallick, \emph{Corks, involutions, and
  {H}eegaard {F}loer homology}, J. Eur. Math. Soc. (JEMS) \textbf{25} (2023),
  no.~6, 2319--2389. \MR{4592871}

\bibitem{dai20222}
Irving Dai, Sungkyung Kang, Abhishek Mallick, JungHwan Park, and Matthew
  Stoffregen, \emph{The $(2, 1) $-cable of the figure-eight knot is not
  smoothly slice}, arXiv preprint arXiv:2207.14187 (2022).

\bibitem{dai2023equivariant}
Irving Dai, Abhishek Mallick, and Matthew Stoffregen, \emph{Equivariant knots
  and knot {F}loer homology}, Journal of Topology \textbf{16} (2023), no.~3,
  1167--1236.

\bibitem{Freedman}
Michael~H. Freedman, \emph{The topology of four-dimensional manifolds}, J.
  Differential Geometry \textbf{17} (1982), no.~3, 357--453.

\bibitem{Gartner}
Michael Gartner, \emph{Projective naturality in {H}eegaard {F}loer homology},
  Algebr. Geom. Topol. \textbf{23} (2023), no.~3, 963--1054. \MR{4598803}

\bibitem{gompf2017infinite}
Robert Gompf, \emph{Infinite order corks}, Geometry \& Topology \textbf{21}
  (2017), no.~4, 2475--2484.

\bibitem{Gompfhandle}
Robert~E. Gompf, \emph{Infinite order corks via handle diagrams}, Algebr. Geom.
  Topol. \textbf{17} (2017), no.~5, 2863--2891. \MR{3704246}

\bibitem{HHL}
Kristen Hendricks, Jennifer Hom, and Tye Lidman, \emph{Applications of
  involutive {H}eegaard {F}loer homology}, J. Inst. Math. Jussieu \textbf{20}
  (2021), no.~1, 187--224. \MR{4205781}

\bibitem{hendricks2022quotient}
Kristen Hendricks, Jennifer Hom, Matthew Stoffregen, and Ian Zemke, \emph{On
  the quotient of the homology cobordism group by {S}eifert spaces.},
  Transactions of the American Mathematical Society, Series B \textbf{9}
  (2022).

\bibitem{HKPS}
Jennifer Hom, Sungkyung Kang, JungHwan Park, and Matthew Stoffregen,
  \emph{Linear independence of rationally slice knots}, Geom. Topol.
  \textbf{26} (2022), no.~7, 3143--3172. \MR{4540903}

\bibitem{HM}
Kristen Hendricks and Ciprian Manolescu, \emph{Involutive {H}eegaard {F}loer
  homology}, Duke Math. J. \textbf{166} (2017), no.~7, 1211--1299.

\bibitem{HMZ}
Kristen Hendricks, Ciprian Manolescu, and Ian Zemke, \emph{A connected sum
  formula for involutive {H}eegaard {F}loer homology}, Selecta Math. (N.S.)
  \textbf{24} (2018), no.~2, 1183--1245.

\bibitem{JTZ}
Andr\'{a}s Juh\'{a}sz, Dylan Thurston, and Ian Zemke, \emph{Naturality and
  mapping class groups in {H}eegard {F}loer homology}, Mem. Amer. Math. Soc.
  \textbf{273} (2021), no.~1338, v+174. \MR{4337438}

\bibitem{JuhaszZemkeSliceDisks}
Andr\'{a}s Juh\'{a}sz and Ian Zemke, \emph{Distinguishing slice disks using
  knot {F}loer homology}, Selecta Math. (N.S.) \textbf{26} (2020), no.~1, Paper
  No. 5, 18. \MR{4045151}

\bibitem{LRS}
Jianfeng Lin, Daniel Ruberman, and Nikolai Saveliev, \emph{On the {F}r\o yshov
  invariant and monopole {L}efschetz number}, J. Differential Geom.
  \textbf{123} (2023), no.~3, 523--593. \MR{4584860}

\bibitem{mallick2022knot}
Abhishek Mallick, \emph{Knot {F}loer homology and surgery on equivariant
  knots}, arXiv preprint arXiv:2201.07299 (2022).

\bibitem{Matveyev}
R.~Matveyev, \emph{A decomposition of smooth simply-connected {$h$}-cobordant
  {$4$}-manifolds}, J. Differential Geom. \textbf{44} (1996), no.~3, 571--582.
  \MR{1431006}

\bibitem{OSKnots}
Peter Ozsv{\'a}th and Zolt{\'a}n Szab{\'o}, \emph{Holomorphic disks and knot
  invariants}, Adv. Math. \textbf{186} (2004), no.~1, 58--116.

\bibitem{OS3manifolds2}
Peter Ozsv{\'a}th and Zolt{\'a}n Szab{\'o}, \emph{Holomorphic disks and three-manifold invariants: properties and
  applications}, Ann. of Math. (2) \textbf{159} (2004), no.~3, 1159--1245.

\bibitem{OS3manifolds1}
Peter Ozsv{\'a}th and Zolt{\'a}n Szab{\'o}, \emph{Holomorphic disks and topological invariants for closed
  three-manifolds}, Ann. of Math. (2) \textbf{159} (2004), no.~3, 1027--1158.

\bibitem{petkovacables}
Ina Petkova, \emph{Cables of thin knots and bordered {H}eegaard {F}loer
  homology}, Quantum Topol. \textbf{4} (2013), no.~4, 377--409.

\bibitem{RasmussenThesis}
Jacob~Andrew Rasmussen, \emph{Floer homology and knot complements}, ProQuest
  LLC, Ann Arbor, MI, 2003, Thesis (Ph.D.)--Harvard University.

\bibitem{RayRuberman}
Arunima Ray and Daniel Ruberman, \emph{Four-dimensional analogues of {D}ehn's
  lemma}, J. Lond. Math. Soc. (2) \textbf{96} (2017), no.~1, 111--132.
  \MR{3687942}

\bibitem{sarkar2015moving}
Sucharit Sarkar, \emph{Moving basepoints and the induced automorphisms of link
  {F}loer homology}, Algebraic \& Geometric Topology \textbf{15} (2015), no.~5,
  2479--2515.

\bibitem{Zemkegraphcobord}
Ian Zemke, \emph{Graph cobordisms and {H}eegaard {F}loer homology}, 2015,
  preprint, arXiv:1512.01184.

\bibitem{Zemkequasistab}
Ian Zemke, \emph{Quasistabilization and basepoint moving maps in link floer
  homology}, Algebraic \& Geometric Topology \textbf{17} (2017), no.~6,
  3461--3518.

\bibitem{Zemkeconnected}
Ian Zemke, \emph{Connected sums and involutive knot {F}loer homology}, Proc.
  Lond. Math. Soc. (3) \textbf{119} (2019), no.~1, 214--265. \MR{3957835}

\bibitem{zemke2019link}
Ian Zemke, \emph{Link cobordisms and functoriality in link {F}loer homology},
  Journal of Topology \textbf{12} (2019), no.~1, 94--220.

\end{thebibliography}
\end{document}